\documentclass[reqno, 12pt]{amsart}
\usepackage{a4wide}
\usepackage{amscd}
\usepackage{enumerate}
\input xy
\xyoption{all}

\newtheorem{thm}{Theorem}[section]

\newtheorem{prop}[thm]{Proposition}

\newtheorem{lm}[thm]{Lemma}

\newtheorem*{mT}{Main Theorem}

\newtheorem*{acknowledgements}{Acknowledgements}

\def \HQ{{Q^m}^*}

\numberwithin{equation}{section}

\newcommand{\Ad}{\mathrm{Ad}}
\newcommand{\tr}{\mathrm{tr}}
\newcommand{\End}{\mathrm{End}}
\newcommand{\pr}{\mathrm{pr}}
\newcommand{\R}{\mathbb{R}}
\newcommand{\C}{\mathbb{C}}

\begin{document}

\title[contact real hypersurfaces]{Contact Real hypersurfaces in the complex hyperbolic quadric }
\author{\textsc{Sebastian Klein and Young Jin Suh}}

\address{Fakult\"at f\"ur Wirtschaftsinformatik und Wirtschaftsmathematik \\
  \newline Seminargeb\"aude A5 \\
  \newline Universit\"at Mannheim \\
  \newline 68131 Mannheim, Germany}
\email{s.klein@math.uni-mannheim.de}
\address{Kyungpook National University \\
\newline College of Natural Sciences\\
\newline Department of Mathematics \& RIRCM\\
\newline Daegu 41566, Republic of Korea}
\email{yjsuh@knu.ac.kr}
\date{}

\begin{abstract}
We give a new proof of the classification of contact real hypersurfaces with constant mean curvature in the complex hyperbolic quadric ${Q^m}^* = SO_{m,2}^o/SO_mSO_2$, where $m\geq 3$.
We show that a contact real hypersurface $M$ in ${Q^m}^*$ for $m\geq 3$ is locally congruent to a tube of radius $r{\in}{\mathbb R}^+$ around the complex hyperbolic quadric ${Q^{m-1}}^*$,
or to a tube of radius \,$r\in\mathbb{R}^+$\, around the \,$\mathfrak A$-principal $m$-dimensional real hyperbolic space
${\mathbb R}H^m$ in ${Q^m}^* = SO_{m,2}^o/SO_mSO_2$, or to a horosphere in ${Q^{m-1}}^*$ induced by a class of $\mathfrak A$-principal geodesics in $\HQ$.
\end{abstract}

\maketitle
\thispagestyle{empty}

\footnote[0]{
This work was supported by grant Proj. No. NRF-2015-R1A2A1A-01002459 from National Research Foundation of Korea.}

\section{Introduction}\label{section 1}

Following Sasaki \cite{Sasaki} and Okumura \cite{O2}, an odd-dimensional, smooth manifold $M^{2m-1}$ is called an \emph{almost contact manifold} if the structure group of its tangent bundle
can be reduced to $U_{m-1} \times 1$ (where $U_{m-1}$ refers to the natural real representation of the unitary group in $m-1$ complex variables). $M^{2m-1}$ is called a \emph{contact manifold}
if there exists a smooth 1-form $\eta$ on $M^{2m-1}$ so that $\eta \wedge d\eta^{m-1} \neq 0$; such an $\eta$ is then called a \emph{contact form} on $M^{2m-1}$.

It was shown by Sasaki \cite[Theorem~5]{Sasaki} that $M^{2m-1}$ is an almost contact manifold if and only if there exists an \emph{almost contact metric structure} $(\phi,\xi,\eta,g)$ on $M^{2m-1}$.
Here $\phi$ is an endomorphism field on $M^{2m-1}$, $\xi$ is a vector field on $M^{2m-1}$, $\eta$ is a 1-form on $M^{2m-1}$ and $g$ is a Riemannian metric on $M^{2m-1}$, and these data are related to
each other in the following way: First we have
$$ \phi^2X=-X+\eta(X)\xi\;,\quad \phi(\xi)=0\;,\quad \eta(\phi X)=0\;,\quad \eta(\xi)=1 $$
for all vector fields $X$ on $M^{2m-1}$, meaning that $(\phi,\xi,\eta)$ is a \emph{almost contact structure}, and moreover this structure is adapted to the Riemannian metric $g$ by
$$ g(\phi X,\phi Y) = g(X,Y)-\eta(X)\,\eta(Y) \quad\text{and}\quad \eta(X)=g(X,\xi) $$
for all vector fields $X,Y$ on $M^{2m-1}$. 

Clearly, if $M^{2m-1}$ has an almost contact metric structure $(\phi,\xi,\eta,g)$ so that  $\eta \wedge d\eta^{m-1} \neq 0$ holds,
then $M^{2m-1}$ is a contact manifold. Conversely, if $M^{2m-1}$ is a contact manifold, then
for any contact form $\eta$ on $M^{2m-1}$ there exists an almost contact metric structure $(\phi,\xi,\eta,g)$ with this $\eta$ by a result due to Sasaki \cite[Theorem~4]{Sasaki}.

Let us now consider a real hypersurface $M$ of a K\"ahler manifold $\bar{M}$ of complex dimension $m$. Then $M$ has real dimension $2m-1$, and the complex structure $J$ and the Riemannian
metric $g$ of $\bar{M}$ induce an almost contact metric structure $(\phi,\xi,\eta,g)$ on $M$: Let $N$ be a unit normal vector field of $M$ in $\bar{M}$, then we choose $\phi$ as the
\emph{structure tensor field} defined by letting $\phi X$ be the $M$-tangential part of $JX$ for any $X\in TM$, choose $\xi$ as the \emph{Reeb vector field} $\xi=-JN$, and choose
$\eta = g(\,\cdot\,,\xi)$. If there exists a smooth, everywhere non-zero function $\rho$ on $M$ in this setting so that
\begin{equation}
\label{eq:intro:contact1}
\mathrm{d}\eta(X,Y) = \rho \cdot g(\phi X,Y)
\end{equation}
holds for all vector field $X,Y$ on $M$, then $M$ is a contact manifold, and $\eta$ a contact form on $M$. In this case $M$ is called a \emph{contact hypersurface} of $\bar{M}$,
see also Blair \cite{BL}, Dragomir and Perrone \cite{DP}. It was noted by Okumura \cite[Equation~(2.13)]{O2} that the condition \eqref{eq:intro:contact1} is equivalent to
\begin{equation}
\label{eq:intro:contact2}
S\phi + \phi S = k\cdot \phi \;,   
\end{equation}  
where $S$ denotes the shape operator of the hypersurface $M$ in $\bar{M}$ with respect to the unit normal vector field $N$, and $k=2\rho$. If the complex dimension of $\bar{M}$ is at least $3$
in this setting, then the function $\rho$ resp.~$k$ is necessarily constant, see \cite[Proposition~2.5]{BerndtSuh}. 

Pursuant to these ideas, the contact hypersurfaces have been classified in the Hermitian symmetric spaces of rank $1$, namely in the complex projective space $\C P^m$ and its non-compact dual, the
complex hyperbolic space $\C H^m$.
Yano and Kon showed in \cite[Theorem~VI.1.5]{YK} that a connected contact hypersurface with constant mean curvature of the complex projective space $\C P^m$ with $m\geq 3$ is 
locally congruent either to a geodesic hypersphere, or to a tube over a real projective space $\R P^n$, $m=2n$, embedded in $\C P^m$ as a totally real,
totally geodesic submanifold. 
Vernon proved in \cite{V} that a complete, connected contact real hypersurface in ${\mathbb C}H^m$ with $m\geq 3$ is congruent to
a tube around a totally geodesic ${\mathbb C}H^{m-1}$ in ${\mathbb C}H^m$, a tube around a real form ${\mathbb R}H^m$ in ${\mathbb C}H^m$,
a geodesic hypersphere in ${\mathbb C}H^m$, or a horosphere in ${\mathbb C}H^m$.
Note that all the contact hypersurfaces in $\C P^m$ or $\C H^m$ are homogeneous and therefore have constant principal curvatures, in particular constant mean curvature. 

When we consider more complicated Hermitian symmetric spaces as ambient space $\bar{M}$, there can be contact hypersurfaces which do not have constant mean curvature.
The class of all contact hypersurfaces $M$ in $\bar{M}$ is very complicated, and a full classification does not appear to be feasible at the present time.
However, if one considers only contact hypersurfaces $M$ with constant mean curvature, the classification problem becomes tractable at least when $\bar{M}$ is a 
Hermitian symmetric space of rank $2$.

The series of irreducible Hermitian symmetric spaces of rank $2$ comprise
the complex quadrics $Q^m = SO_{m+2}/SO_2 SO_m$ (isomorphic
to the real 2-Grassmannians $G_2^+(\R^{m+2})$ of oriented planes in $\R^{m+2}$), the complex 2-Grassmannians $G_2(\C^{m+2})=SU_{m+2}/S(U_2 U_m)$, and their non-compact duals, the complex hyperbolic quadrics
${Q^m}^* = SO_{2,m}^o/SO_2 SO_m$ and the duals of the complex 2-Grassmannians $G_2^*(\C^{m+2}) = SU_{2,m}/S(U_2 U_m)$.

The classification of contact hypersurfaces with constant mean curvature
in the complex quadric $Q^m$ and in its non-compact dual ${Q^m}^*$ has been carried out by Berndt and the second author of the present paper in \cite{BerndtSuh}. The result
of the classification is that a contact hypersurface with constant mean curvature in $Q^m$ is congruent to an open part of a tube of radius $0<r<\tfrac{\pi}{2\sqrt{2}}$ around the
$m$-dimensional sphere $S^m$ embedded in $Q^m$ as a real form. A contact hypersurface with constant mean curvature in ${Q^m}^*$ is either congruent to an open part of a tube of radius $r>0$ around
the $m$-dimensional hyperbolic space $\R H^m$ embedded in ${Q^m}^*$, or to an open part of a tube of radius $r>0$ around ${Q^{m-1}}^*$ embedded in ${Q^m}^*$ as a complex, totally geodesic submanifold,
or to an open part of a horosphere of certain position in ${Q^m}^*$. 

In the complex 2-Grassmannians $G_2(\C^{m+2})$, the contact hypersurfaces with constant mean curvature have also been classified by the second author of the present paper in  \cite{SB}.
He shows that such a hypersurface is congruent to an open part of a tube around a totally geodesic quaternionic projective space $\mathbb{H}P^n$ in $G_2(\C^{m+2})$, where $m=2n$.
For the non-compact dual $G_2^*(\C^{m+2})$ of these Grassmannians, as far as we know there does not exist a classification of contact hypersurfaces with constant mean curvature.
However, Berndt, Lee and Suh \cite{BerndtLeeSuh} classified
contact hypersurfaces of $G_2^*(\C^{m+2})$ which satisfy another curvature condition, namely that the principal curvature function $\alpha$ corresponding to the Reeb
vector field of the hypersurface is constant. The result of the classification is that any such hypersurface of $G_2^*(\C^{m+2})$ is congruent either to an open part of a tube
around a totally geodesic quaternionic hyperbolic space $\mathbb{H}H^n$ in $G_2(\C^{m+2})$ (only if $m=2n$ is even), or to an open part of a horosphere in a certain position in $G_2^*(\C^{m+2})$.
Note that all these hypersurfaces have constant mean curvature (this follows from \cite[Proposition~4.1, Proposition~4.2(ii)]{BerndtLeeSuh}). 

In \cite{SQ},
the second author has given a different proof of the classification of contact hypersurfaces with constant mean curvature for the complex quadric in $Q^m$, using different methods based to a larger extent
on the characterization of Equation~\eqref{eq:intro:contact2}.
It is the objective of the present paper
to apply the methods of \cite{SQ} to the classification in ${Q^m}^*$, to re-obtain the classification result which was first shown in \cite{BerndtSuh} also for this symmetric space. In this way we will prove:

\begin{mT}\label{Main Theorem}
Let M be a connected orientable real hypersurface with constant
mean curvature in the complex hyperbolic quadric $\HQ$, $m{\ge}3$. Then $M$ is a contact hypersurface if and only if $M$ is congruent to an open
part of one of the following contact hypersurfaces in $\HQ$:
\par
(i) the tube of radius $r>0$ around the complex hyperbolic quadric ${Q^{m-1}}^*$
which is embedded in $\HQ$ as a totally geodesic complex hypersurface;
\par
(ii) a horosphere in $\HQ$ whose center at infinity is the equivalence class of an
$\mathfrak A$-principal geodesic in $\HQ$;
\par
(iii) the tube of radius $r>0$ around the $n$-dimensional real hyperbolic space
${\mathbb R}H^n$ which is embedded in $\HQ$ as a real space form of $\HQ$.
\end{mT}

The proof is based on the use of the ``fundamental geometric structures'' of the complex hyperbolic quadric ${Q^m}^*$, which are introduced in Section~\ref{section 2} of the present paper.
The term ``$\mathfrak{A}$-principal'' that occurs in part (ii) of the Main Theorem refers to a specific orbit of the isotropy action on the tangent space of ${Q^m}^*$ and is also explained
in detail in Section~\ref{section 2}. 
Section~\ref{section 3} contains a discussion of the tubes around totally geodesic
submanifolds of ${Q^m}^*$ that are involved in parts (i) and (iii) of the Main Theorem. In Section~\ref{section 4}, some general results and formulas concerning real hypersurfaces
of ${Q^m}^*$ are derived, which are then used in Section~\ref{section 5} to complete our proof of the Main Theorem.

\medskip

\section{The complex hyperbolic quadric}\label{section 2}

The \,$m$-dimensional complex hyperbolic quadric ${Q^m}^*$ is the non-compact dual of the $m$-dimensional complex quadric $Q^m$, i.e.~the simply connected Riemannian symmetric
space whose curvature tensor is the negative of the curvature tensor of $Q^m$.

The complex hyperbolic quadric ${Q^m}^*$ cannot be realized as a homogeneous complex hypersurface of the complex hyperbolic space ${\mathbb C}H^{m+1}$. In fact, Smyth \cite[Theorem~3(ii)]{BS2} has shown
that every homogeneous complex hypersurface in ${\mathbb C}H^{m+1}$ is totally geodesic. This is in marked contrast to the situation for the complex quadric $Q^m$, which can be realized
as a homogeneous complex hypersurface of the complex projective space ${\mathbb C}P^{m+1}$ in such a way that the shape operator for any unit normal vector to $Q^m$ is a real structure on the corresponding
tangent space of $Q^m$, see \cite{R} and \cite{K}.
Another related result by Smyth, \cite[Theorem~1]{BS2}, which states that any complex hypersurface ${\mathbb C}H^{m+1}$ for which
the square of the shape operator has constant eigenvalues (counted with multiplicity) is totally geodesic, also precludes the possibility of a model of ${Q^m}^*$ as a complex
hypersurface of ${\mathbb C}H^{m+1}$ with the analogous property for the shape operator.

Therefore we realize the complex hyperbolic quadric ${Q^m}^*$ as the quotient manifold $SO_{2,m}/SO_2 SO_m$. As ${Q^1}^*$ is isomorphic to the real hyperbolic space
$\mathbb{R}H^2 = SO_{1,2}/SO_2$, and ${Q^2}^*$ is isomorphic to the Hermitian product of complex hyperbolic spaces $\mathbb{C}H^1 \times \mathbb{C}H^1$, we suppose $m\geq 3$
in the sequel and throughout this paper. Let $G:= SO_{2,m}$ be the transvection group of ${Q^m}^*$ and $K := SO_2 SO_m$ be the isotropy group of ${Q^m}^*$ at the ``origin''
$p_0 := eK \in {Q^m}^*$. Then
$$ \sigma: G \to G,\; g \mapsto sgs^{-1} \quad\text{with}\quad s := \left( \begin{smallmatrix} -1 & & & & & \\ & -1 & & & & \\ & & 1 & & & \\ & & & 1 & &  \\ & & & & \ddots & \\ & & & & & 1 \end{smallmatrix} \right) $$
is an involutive Lie group automorphism of $G$ with $\mathrm{Fix}(\sigma)_0 = K$, and therefore ${Q^m}^* = G/K$ is a Riemannian symmetric space. The center of the isotropy group $K$ is
isomorphic to $SO_2$, and therefore ${Q^m}^*$ is in fact a Hermitian symmetric space.

The Lie algebra $\mathfrak{g} := \mathfrak{so}_{2,m}$ of $G$ is given by
$$ \mathfrak{g} = \bigr\{ X \in \mathfrak{gl}(m+2,\mathbb{R}) \bigr| X^t \cdot s = -s \cdot X \bigr\} $$
(see \cite[p.~59]{Kna}). In the sequel we will write members of $\mathfrak{g}$ as block matrices with respect to the decomposition $\mathbb{R}^{m+2}=\mathbb{R}^2 \oplus \mathbb{R}^m$, i.e.~in the form
$$ X = \left( \begin{smallmatrix} X_{11} & X_{12} \\ X_{21} & X_{22} \end{smallmatrix} \right) \;, $$
where $X_{11}$, $X_{12}$, $X_{21}$, $X_{22}$ are real matrices of the dimension $2\times 2$, $2\times m$, $m\times 2$ and $m\times m$, respectively. Then
$$ \mathfrak{g} = \left\{ \; \left. \left( \begin{smallmatrix} X_{11} & X_{12} \\ X_{21} & X_{22} \end{smallmatrix} \right) \;\right|\; X_{11}^t=-X_{11}, \; X_{12}^t = X_{21},\; X_{22}^t = -X_{22} \;\right\} \; . $$
The linearisation \,$\sigma_L=\mathrm{Ad}(s): \mathfrak{g}\to\mathfrak{g}$\, of the involutive Lie group automorphism \,$\sigma$\, induces the Cartan decomposition $\mathfrak{g} = \mathfrak{k}
\oplus \mathfrak{m}$, where the Lie subalgebra
$$  \mathfrak{k} = \mathrm{Eig}(\sigma_*,1) = \{ X \in \mathfrak{g} | sXs^{-1}=X\} \\
  = \left\{ \; \left. \left( \begin{smallmatrix} X_{11} & 0 \\ 0 & X_{22} \end{smallmatrix} \right) \;\right|\; X_{11}^t=-X_{11}, \; X_{22}^t = -X_{22} \;\right\} \cong \mathfrak{so}_2 \oplus \mathfrak{so}_m$$
is the Lie algebra of the isotropy group $K$, and the $2m$-dimensional linear subspace
$$  \mathfrak{m} = \mathrm{Eig}(\sigma_*,-1) = \{ X \in \mathfrak{g} | sXs^{-1}=-X\} 
  = \left\{ \; \left. \left( \begin{smallmatrix} 0 & X_{12} \\ X_{21} & 0 \end{smallmatrix} \right) \;\right|\; X_{12}^t = X_{21} \;\right\} $$
is canonically isomorphic to the tangent space $T_{p_0}{Q^m}^*$. Under the identification $T_{p_0}{Q^m}^* \cong \mathfrak{m}$, the Riemannian metric $g$ of ${Q^m}^*$ (where the constant factor
of the metric is chosen so that the formulae become as simple as possible) is given by
$$ g(X,Y) = \tfrac12\,\tr(Y^t \cdot X) = \tr(Y_{12}\cdot X_{21}) \quad\text{for}\quad X,Y \in \mathfrak{m} \; . $$
$g$ is clearly $\Ad(K)$-invariant, and therefore corresponds to an $\Ad(G)$-invariant Riemannian metric on ${Q^m}^*$.
The complex structure $J$ of the Hermitian symmetric space is given by
$$ JX = \Ad(j)X \quad\text{for}\quad X \in \mathfrak{m}, \quad \text{where}\quad j := \left( \begin{smallmatrix} 0 & 1 & & & & \\ -1 & 0 & & & & \\ & & 1 & & & \\ & & & 1 & & \\ & & & & \ddots & \\ & & & & & 1
\end{smallmatrix} \right) \in K \; . $$
Because $j$ is in the center of $K$, the orthogonal linear map $J$ is $\Ad(K)$-invariant, and thus defines an $\Ad(G)$-invariant Hermitian structure on ${Q^m}^*$. By identifying the
multiplication with the unit complex number $i$ with the application of the linear map $J$, the tangent spaces of ${Q^m}^*$ thus become $m$-dimensional complex linear spaces,
and we will adopt this point of view in the sequel. 

Like for the complex quadric (again compare \cite{R} and \cite{K}), there is another important structure on the tangent bundle of the complex quadric besides the Riemannian metric and the complex structure,
namely an $S^1$-bundle $\mathfrak{A}$ of real structures (conjugations). The situation here differs from that of the complex quadric in that for ${Q^m}^*$, the real structures in $\mathfrak{A}$ cannot be
interpreted as the shape operator of a complex hypersurface in a complex space form, but as the following considerations will show, $\mathfrak{A}$ still plays a fundamental role in the description
of the geometry of ${Q^m}^*$. 

Let
$$ a_0 := \left( \begin{smallmatrix} 1 & & & & & \\ & -1 & & & & \\ & & 1 & & & \\ & & & 1 & &  \\ & & & & \ddots & \\ & & & & & 1 \end{smallmatrix} \right) \; . $$
Note that we have $a_0 \not\in K$, but only $a_0 \in O_2\,SO_m$. However, $\Ad(a_0)$ still leaves $\mathfrak{m}$ invariant, and therefore defines an $\mathbb{R}$-linear map $A_0$ on the tangent space
$\mathfrak{m} \cong T_{p_0}{Q^m}^*$. $A_0$ turns out to be an involutive orthogonal map with $A_0 \circ J = -J \circ A_0$\, (i.e.~$A_0$ is anti-linear with respect to the complex structure of
$T_{p_0}{Q^m}^*$), and hence a real structure on $T_{p_0}{Q^m}^*$. But $A_0$ commutes with $\Ad(g)$ not for all $g \in K$, but only for $g \in SO_m \subset K$. More specifically, for $g=(g_1,g_2) \in K$ with
$g_1 \in SO_2$ and $g_2 \in SO_m$, say $g_1 = \left( \begin{smallmatrix} \cos(t) & -\sin(t) \\ \sin(t) & \cos(t) \end{smallmatrix} \right)$ with $t \in \R$ (so that $\Ad(g_1)$ 
corresponds to multiplication with the complex number $\mu := e^{it}$), we have
$$ A_0 \circ \Ad(g) = \mu^{-2} \cdot \Ad(g) \circ A_0 \; . $$
This equation shows that the object which is \,$\Ad(K)$-invariant and therefore geometrically relevant is not the real structure $A_0$ by itself, but rather the ``circle of real structures''
$$ \mathfrak{A}_{p_0} := \{ \lambda\,A_0 | \lambda \in S^1 \} \; . $$
$\mathfrak{A}_{p_0}$ is $\Ad(K)$-invariant, and therefore generates an $\Ad(G)$-invariant $S^1$-subbundle $\mathfrak{A}$ of the endomorphism bundle $\mathrm{End}(T{Q^m}^*)$, consisting
of real structures (conjugations) on the tangent spaces of ${Q^m}^*$. For any $A \in \mathfrak{A}$, the tangent line to the fibre of $\mathfrak{A}$ through $A$ is spanned by $JA$. 

For any $p\in {Q^m}^*$ and $A \in \mathfrak{A}_p$, the real structure $A$ induces a splitting 
$$ T_p{Q^m}^* = V(A) \oplus JV(A) $$
into two orthogonal, maximal totally real subspaces of the tangent space $T_p{Q^m}^*$. Here $V(A)$ resp.~$JV(A)$ are the $(+1)$-eigenspace resp.~the $(-1)$-eigenspace of $A$. For every
unit vector $Z \in T_p{Q^m}^*$ there exist $t\in [0,\tfrac\pi4]$, $A \in \mathfrak{A}_p$ and orthonormal vectors $X,Y \in V(A)$ so that
$$ Z = \cos(t)\cdot X + \sin(t)\cdot JY $$
holds; see \cite[Proposition~3]{R}. Here $t$ is uniquely determined by $Z$. The vector $Z$ is singular, i.e.~contained in more than one Cartan subalgebra of $\mathfrak{m}$, if and only if
either $t=0$ or $t=\tfrac\pi4$ holds. The vectors with $t=0$ are called \emph{$\mathfrak{A}$-principal}, whereas the vectors with $t=\tfrac\pi4$ are called \emph{$\mathfrak{A}$-isotropic}.
If $Z$ is regular, i.e.~$0<t<\tfrac\pi4$ holds, then also $A$ and $X,Y$ are uniquely determined by $Z$. 

Like for the complex quadric, the Riemannian curvature tensor $R$ of ${Q^m}^*$ can be fully described in terms of the ``fundamental geometric structures'' $g$, $J$ and $\mathfrak{A}$.
In fact, under the correspondence $T_{p_0}{Q^m}^* \cong \mathfrak{m}$, the curvature $R(X,Y)Z$ corresponds to $-[[X,Y],Z]$ for $X,Y,Z \in \mathfrak{m}$, see \cite[Chapter~XI, Theorem~3.2(1)]{KO}. 
By evaluating the latter expression
explicitly, one can show that one has
\begin{eqnarray*}
R(X,Y)Z & = &-g(Y,Z)X + g(X,Z)Y \\
& & - \, g(JY,Z)JX + g(JX,Z)JY + 2g(JX,Y)JZ \\
 & & -\, g(AY,Z)AX + g(AX,Z)AY \\
& & - \,g(JAY,Z)JAX + g(JAX,Z)JAY
\end{eqnarray*}
for arbitrary $A \in \mathfrak{A}_{p_0}$. Therefore the curvature of ${Q^m}^*$ is the negative of that of the complex quadric $Q^m$, compare \cite[Theorem~1]{R}. This confirms that the
symmetric space ${Q^m}^*$ which we have constructed here is indeed the non-compact dual of the complex quadric.

As Nomizu \cite[Theorem~15.3]{Nomizu} has shown, there exists one and only one torsion-free covariant derivative $\bar{\nabla}$ on ${Q^m}^*$ so that the symmetric involutions $s_p: {Q^m}^* \to {Q^m}^*$ at $p\in {Q^m}^*$ are all affine. $\bar{\nabla}$ is the \emph{canonical covariant derivative} of ${Q^m}^*$. With respect to $\bar{\nabla}$, the action of any member of $G$ on ${Q^m}^*$ is also affine.
Moreover, $\bar{\nabla}$ is the Levi-Civita connection corresponding to the Riemannian metric $g$, and therefore $g$ is parallel with respect to $\bar{\nabla}$. Moreover,
it is well-known that ${Q^m}^*$ becomes a K\"ahler manifold in this way, i.e.~the complex structure $J$ is also parallel. Finally, because the $S^1$-subbundle $\mathfrak{A}$ of the
endomorphism bundle $\End(T{Q^m}^*)$ is $\Ad(G)$-invariant, it is also
parallel with respect to the covariant derivative $\bar{\nabla}^{\End}$ induced by $\bar{\nabla}$ on $\End(T{Q^m}^*)$. Because the tangent line of the fiber of $\mathfrak{A}$ through some
$A_p \in \mathfrak{A}$\, is spanned by $JA_p$, this means precisely that
for any section $A$ of $\mathfrak{A}$ there exists a real-valued 1-form $q: T{Q^m}^* \to \R$ so that
$$ \bar{\nabla}^{\End}_v A = q(v)\cdot JA_p \quad \text{holds for $p\in {Q^m}^*$, $v \in T_p{Q^m}^*$.} $$

\par
\vskip 8pt

\section {The totally geodesic submanifolds ${Q^{m-1}}^* \subset {Q^m}^*$ and ${\mathbb{R}}H^m \subset {Q^m}^*$}\label{section 3}

The obvious embedding of Lie groups $SO_{2,m-1}^o \to SO_{2,m}^o$ induces a totally geodesic embedding of ${Q^{m-1}}^*=SO_{2,m-1}/SO_2SO_{m-1}$ into ${Q^m}^*=SO_{2,m}/SO_2SO_m$.
We will view ${Q^{m-1}}^*$ as a totally geodesic complex hypersurface
of ${Q^m}^*$ by means of this embedding.

From the construction of the $S^1$-subbundle ${\mathfrak A}$ of $\End(T{Q^m}^*)$ of real structures,
it is clear that both the tangent space $T_p{Q^{m-1}}^*$,
and the normal space $\nu_p {Q^{m-1}}^*$ of ${Q^{m-1}}^*$ in ${Q^m}^*$ are $A$-invariant for every \,$p\in {Q^{m-1}}^*$ and every real structure $A \in {\mathfrak A}_p$.
Because $\nu_p {Q^{m-1}}^*$ is complex-1-dimensional, there exists for any unit normal vector $N \in \nu_p {Q^{m-1}}^*$ some $A \in \mathfrak{A}_p$ with $N \in V(A)$. We then have
\[ T_p{Q^{m-1}}^* = (V(A) \ominus {\mathbb R}N) \oplus J(V(A) \ominus {\mathbb R}N). \]

We are now going to calculate the principal curvatures and principal curvature spaces of the tube with radius $r$ around ${Q^{m-1}}^*$ in $\HQ$.  The normal Jacobi operator $R_N$ leaves the tangent space $T_p{Q^{m-1}}^*$ and the normal space $\nu_p{Q^{m-1}}^*$ invariant. When restricted to $T_p{Q^{m-1}}^*$, the eigenvalues of $R_N$ are $0$ and $-2$ with corresponding eigenspaces $J(V(A) \ominus {\mathbb R}N)$ and $V(A) \ominus {\mathbb R}N$, respectively. The corresponding principal curvatures on the tube of radius $r$ are $0$ and $-\sqrt{2}\tanh(\sqrt{2}r)$, and the corresponding principal curvature spaces are the parallel translates of $J(V(A) \ominus {\mathbb R}N)$ and $V(A) \ominus {\mathbb R}N$ along the geodesic $\gamma$ in ${Q^m}^*$ with $\gamma(0) = p$ and $\dot{\gamma}(0) = N$ from $\gamma(0)$ to $\gamma(r)$. When restricted to $\nu_p {Q^{m-1}}^* \ominus {\mathbb R}N$, the only eigenvalue of $R_N$ is $-2$ with corresponding eigenspace ${\mathbb R}JN$. The corresponding principal curvature on the tube of radius $r$ is $-\sqrt{2}\coth(\sqrt{2}r)$, and the corresponding principal curvature space is the parallel translate of ${\mathbb R}JN$ along $\gamma$ from $\gamma(0)$ to $\gamma(r)$. For all $r>0$\, this process leads to real hypersurfaces in ${Q^m}^*$. Therefore the quadric ${Q^{m-1}}^*$ is the only singular orbit of the cohomogeneity one action of $SO_{m-1,2}^o \subset SO_{m,2}^o$ on ${Q^m}^*$. The principal orbits are the tubes around this singular orbit.
It follows that the tube of radius $r$ around ${Q^{m-1}}^*$ has three distinct constant principal curvatures $0$, $-\sqrt{2}\tanh(\sqrt{2}r)$, $-\sqrt{2}\coth(\sqrt{2}r)$.

\begin{prop}\label{proposition 3.1}
Let $M$ be the tube of radius $r>0$ around the totally geodesic ${Q^{m-1}}^*$ in $\HQ$. Then the following statements hold:
\begin{itemize}
\item[1.] $M$ is a Hopf hypersurface.
\item[2.] Every unit normal vector $N$ of $M$ is ${\mathfrak A}$-principal and therefore there exists $A \in {\mathfrak A}$ such that $AN = N$.
\item[3.] If we choose the sign of $N$ such that $N$ points inwards (i.e.~towards the focal surface ${Q^{m-1}}^*$), the principal curvatures and corresponding principal curvature spaces of $M$ are
\begin{center}
\begin{tabular}{|l|l|l|}
\hline
\mbox{principal curvature} & \mbox{eigenspace}  & \mbox{multiplicity}\\
\hline
${\lambda}=0$ & $J(V(A) \ominus {\mathbb R}N)$ & $m-1$ \\
${\mu}=\sqrt{2}\tanh(\sqrt{2}r)$ & $V(A) \ominus {\mathbb R}N$ & $m-1$\\
${\alpha}=\sqrt{2}\coth(\sqrt{2}r)$ & ${\mathbb R}JN$ & $1$ \\
\hline
\end{tabular}
\end{center}
\item[4.] The shape operator $S$ and the structure tensor field $\phi$ ($\phi = \pr_{TM} \circ J$, where $\pr_{TM}: T{Q^m}^* \to TM$ denotes the orthogonal projection) satisfy
\[S\phi + \phi S = \sqrt{2}\tanh(\sqrt{2}r)\phi.\]
\end{itemize}
\end{prop}

A similar discussion applies to the totally geodesic $\R H^m$ in ${Q^m}^*$. The obvious embedding of Lie groups $SO_{1,m} \to SO_{2,m}$ induces a totally geodesic embedding of
$\mathbb{R} H^m = SO_{1,m}/SO_m$ into ${Q^m}^* = SO_{2,m}/SO_2SO_m$. We will view $\R H^m$ as a real form, i.e.~a totally geodesic, totally real, real-$m$-dimensional submanifold of ${Q^m}^*$
by means of this embedding. The totally geodesic submanifold $\R H^m$ of ${Q^m}^*$ is $\mathfrak{A}$-principal, i.e.~for every \,$p\in \R H^m$ there exists $A\in\mathfrak{A}_p$ so that
$T_p\R H^m = V(A)$ holds. Then the normal space of $\R H^m$ in ${Q^m}^*$ at $p$ is given by $\nu_p \R H^m = JV(A)$.

We will again calculate the principal curvatures and principal curvature spaces of the tube with radius $r>0$ around $\mathbb{R}H^m$ in $\HQ$ with respect to the unit normal vector $N \in JV(A)$.
The normal Jacobi operator $R_N$ leaves the tangent space $V(A)$ and the normal space $JV(A)$ invariant. When restricted to $T_p\mathbb{R}H^m$, the eigenvalues of $R_N$ are
$0$ and $-2$, and the corresponding eigenspaces are $V(A) \ominus \mathbb{R} JN$ and $\mathbb{R} JN$, respectively. The corresponding principal curvatures on the tube of radius $r$ are
$0$ and $-\sqrt{2}\,\tanh(\sqrt{2}r)$, and the corresponding principal curvature spaces are the parallel translates of $V(A) \ominus \mathbb{R} JN$ and $\mathbb{R} JN$ along the geodesic $\gamma$
in ${Q^m}^*$ with $\gamma(0) = p$ and $\dot{\gamma}(0) = N$ from $\gamma(0)$ to $\gamma(r)$. When restricted to $\nu_p \mathbb{R}H^m \ominus {\mathbb R}N$, the only eigenvalue of $R_N$ is
$-2$ with corresponding eigenspace $JV(A) \ominus \mathbb{R}N$. The corresponding principal curvature on the tube of radius $r$ is $-\sqrt{2}\,\coth(\sqrt{2}r)$, and
the corresponding principal curvature space is the parallel translate of $JV(A) \ominus \mathbb{R}N$ along $\gamma$ from $\gamma(0)$ to $\gamma(r)$. For all $r>0$ this process leads to
real hypersurfaces in ${Q^m}^*$. Therefore $\mathbb{R}H^m$ is the only singular orbit of the cohomogeneity one action of $SO_{m,1}^o \subset SO_{m,2}^o$ on ${Q^m}^*$.
It follows that the tube of radius $r$ around $\mathbb{R}H^m$ has three distinct constant principal curvatures $0$, $-\sqrt{2}\tanh(\sqrt{2}r)$, $-\sqrt{2}\coth(\sqrt{2}r)$.

\begin{prop}\label{proposition 3.2}
Let $M$ be the tube of radius $r>0$ around the totally geodesic $\mathbb{R}H^m$ in $\HQ$. Then the following statements hold:
\begin{itemize}
\item[1.] $M$ is a Hopf hypersurface.
\item[2.] Every unit normal vector $N$ of $M$ is contained in $JV(A)=V(-A)$, i.e.~we have $AN=-N$.
\item[3.] If we choose the sign of $N$ such that $N$ points inwards (i.e.~towards the focal surface $\mathbb{R}H^m$), the principal curvatures and corresponding principal curvature spaces of $M$ are
\begin{center}
\begin{tabular}{|l|l|l|}
\hline
\mbox{principal curvature} & \mbox{eigenspace}  & \mbox{multiplicity}\\
\hline
${\lambda}=0$ & $J(V(-A) \ominus \mathbb{R}N)$ & $m-1$ \\
${\mu}=\sqrt{2}\coth(\sqrt{2}r)$ & $V(-A) \ominus {\mathbb R}N$ & $m-1$\\
${\alpha}=\sqrt{2}\tanh(\sqrt{2}r)$ & ${\mathbb R}JN$ & $1$ \\
\hline
\end{tabular}
\end{center}
\item[4.] The shape operator $S$ and the structure tensor field $\phi$ ($\phi = \pr_{TM} \circ J$, where $\pr_{TM}: T{Q^m}^* \to TM$ denotes the orthogonal projection) satisfy
\[S\phi + \phi S = \sqrt{2}\coth(\sqrt{2}r)\phi.\]
\end{itemize}
\end{prop}

\section{Some general equations}\label{section 4}

Let $M$ be a  real hypersurface of ${Q^m}^*$ and denote by $(\phi,\xi,\eta,g)$ the induced almost contact metric structure. Note that $\xi = -JN$, where $N$ is a (local) unit normal vector field of $M$. The tangent bundle $TM$ of $M$ splits orthogonally into  $TM = {\mathcal C} \oplus {\mathbb R}\xi$, where ${\mathcal C} = {\rm ker}(\eta)$ is the maximal complex subbundle of $TM$. The structure tensor field $\phi$ restricted to ${\mathcal C}$ coincides with the complex structure $J$ restricted to ${\mathcal C}$, and $\phi \xi = 0$.

At each point $z \in M$ we also define the maximal ${\mathfrak A}$-invariant subspace of $T_zM$
\[
{\mathcal Q}_z = \{X \in T_zM \mid AX \in T_zM\ {\rm for\ all}\ A \in {\mathfrak A}_z\}.
\]

\begin{lm}\label{lemma 4.1}
For each $z \in M$ we have
\begin{itemize}
\item[(i)] If $N_z$ is ${\mathfrak A}$-principal, then ${\mathcal Q}_z = {\mathcal C}_z$.
\item[(ii)] If $N_z$ is not ${\mathfrak A}$-principal, there exist a conjugation $A \in {\mathfrak A}_z$ and orthonormal vectors $X,Y \in V(A)$ such that $N_z = \cos(t)X + \sin(t)JY$ for some $t \in (0,\pi/4]$.
Then we have ${\mathcal Q}_z = {\mathcal C}_z \ominus {\mathbb C}(JX + Y)$.
\end{itemize}
\end{lm}

 \begin{proof} First suppose that $N_z$ is ${\mathfrak A}$-principal. Then there exists a conjugation $A \in {\mathfrak A}_z$ such that $N_z \in V(A_z)$, that is, $AN_z = N_z$. We thus have $A\xi_z = -AJN_z = JAN_z = JN_z = -\xi_z$. It follows that $A$ restricted to ${\mathbb C}N_z$ is the orthogonal reflection in the line ${\mathbb R}N_z$. Because all conjugations in ${\mathfrak A}_z$ differ just by a rotation on such planes, we see that ${\mathbb C}N_z$ is invariant under ${\mathfrak A}_z$. This implies that ${\mathcal C}_z = T_z{Q^m}^* \ominus {\mathbb C}N_z$ is invariant under ${\mathfrak A}_z$, and hence ${\mathcal Q}_z = {\mathcal C}_z$ holds.

Now suppose that $N_z$ is not ${\mathfrak A}$-principal. By a result due to Reckziegel \cite[Proposition~3]{R} there exist $t \in (0,\pi/4]$, 
a conjugation $A \in {\mathfrak A}_z$ and orthonormal vectors $X,Y \in V(A)$ so that $N_z = \cos(t)X + \sin(t)JY$ holds. The conjugation $A$ restricted to ${\mathbb C}X \oplus {\mathbb C}Y$ is just the orthogonal reflection in ${\mathbb R}X \oplus {\mathbb R}Y$. Again, since all conjugations in ${\mathfrak A}_z$ differ just by a rotation on such invariant spaces we see that ${\mathbb C}X \oplus {\mathbb C}Y$ is invariant under ${\mathfrak A}_z$. This implies that $T_z{Q^m}^* \ominus ({\mathbb C}X \oplus {\mathbb C}Y) = {\mathcal C}_z \ominus {\mathbb C}(JX+Y)$ is invariant under ${\mathfrak A}_z$, and hence ${\mathcal Q}_z = {\mathcal C}_z \ominus {\mathbb C}(JX+Y)$ holds.
\end{proof}

\medskip
We see from the previous lemma that the rank of the distribution ${\mathcal Q}$ is generally not constant on $M$. However, if $N_z$ is not ${\mathfrak A}$-principal for some $z\in M$,
then $N$ is also not ${\mathfrak A}$-principal on an open neighborhood of $z$, and therefore ${\mathcal Q}$ then is a regular distribution on that open neighborhood of $z$.

\medskip
We are interested in real hypersurfaces $M$ for which both the maximal complex subbundle ${\mathcal C}$ and the  maximal $\mathfrak{A}$-invariant subbundle ${\mathcal Q}$ of $TM$
are invariant under the shape operator $S$ of $M$. A real hypersurface $M$ of a K\"ahler manifold
is called a \emph{Hopf hypersurface} if $\mathcal{C}$ is invariant under the shape operator of $M$. One can show that $M$ is Hopf if and only if the Reeb flow on $M$,
i.e.~the flow of the Reeb vector field $\xi=-JN$, is geodesic. An equivalent condition is that
the shape operator $S$ of $M$ in $\HQ$ satisfies
$$ S\xi = \alpha \xi $$
with the smooth function $\alpha = g(S\xi,\xi)$ on $M$. Note that then $\xi$ is a principal curvature vector field corresponding to the principal curvature function $\alpha$.

\begin{lm}\label{lemma 4.2}
Let $M$ be a Hopf hypersurface in the complex hyperbolic quadric $\HQ$ with (local) unit normal vector field $N$. For each point in $z \in M$ we choose $A \in {\mathfrak A}_z$ such that
$N_z = \cos(t)Z_1 + \sin(t)JZ_2$ holds
for some orthonormal vectors $Z_1,Z_2 \in V(A)$ and $0 \leq t \leq \frac{\pi}{4}$. Then
\begin{eqnarray*}
0 & = & 2g(S \phi SX,Y) - \alpha g((\phi S + S\phi)X,Y) +  2g(\phi X,Y) \\
& & - 2g(X,AN)g(Y,A\xi) + 2g(Y,AN)g(X,A\xi) \\
& &  - 2g(\xi,A\xi) \{g(Y,AN)\eta(X) - g(X,AN)\eta(Y) \}
\end{eqnarray*}
holds for all vector fields $X$ and $Y$ on $M$.
\end{lm}

\begin{proof}
For any vector field $X$ on $M$ in $\HQ$, we may decompose $JX$ as
$$JX={\phi}X+{\eta}(X)N$$
using a unit normal vector field $N$ to $M$. In this situation, the Codazzi equation states
\begin{equation*}
\begin{split}
g((\nabla_XS)Y - (\nabla_YS)X,Z) & =  -\eta(X)g(\phi Y,Z) + \eta(Y) g(\phi X,Z) + 2\eta(Z) g(\phi X,Y) \\
& \quad \ \   - g(X,AN)g(AY,Z) + g(Y,AN)g(AX,Z)\\
& \quad \ \   - g(X,A\xi)g(J AY,Z) + g(Y,A\xi)g(JAX,Z).
\end{split}
\end{equation*}
By putting $Z = \xi$ we get
\begin{equation*}
\begin{split}
g((\nabla_XS)Y - (\nabla_YS)X,\xi) & =    2 g(\phi X,Y) \\
& \quad \ \  - g(X,AN)g(Y,A\xi) + g(Y,AN)g(X,A\xi)\\
& \quad \ \  + g(X,A\xi)g(JY,A\xi) - g(Y,A\xi)g(JX,A\xi).
\end{split}
\end{equation*}
On the other hand, we have
\begin{eqnarray*}
 & & g((\nabla_XS)Y - (\nabla_YS)X,\xi) \\
& = & g((\nabla_XS)\xi,Y) - g((\nabla_YS)\xi,X) \\
& = & (X\alpha)\eta(Y) - (Y\alpha)\eta(X) + \alpha g((S\phi + \phi
S)X,Y) - 2g(S \phi SX,Y).
\end{eqnarray*}
Comparing the previous two equations and putting $X = \xi$ yields
$$
Y\alpha  =  (\xi \alpha)\eta(Y)  - 2g(\xi,AN)g(Y,A\xi) +
2g(Y,AN)g(\xi,A\xi).
$$
Inserting this into the previous equation gives
\begin{eqnarray*}
 & & g((\nabla_XS)Y - (\nabla_YS)X,\xi) \\
& = &  2g(\xi,AN)g(X,A\xi)\eta(Y) - 2g(X,AN)g(\xi,A\xi)\eta(Y) \\
& &  - 2g(\xi,AN)g(Y,A\xi)\eta(X) + 2g(Y,AN)g(\xi,A\xi)\eta(X) \\
& & + \alpha g((\phi S + S\phi)X,Y) - 2g(S \phi SX,Y) .
\end{eqnarray*}
Altogether this implies
\begin{eqnarray}
0 & = & 2g(S \phi SX,Y) - \alpha g((\phi S + S\phi)X,Y) + 2 g(\phi X,Y) \notag \\
& & - g(X,AN)g(Y,A\xi) + g(Y,AN)g(X,A\xi) \notag \\
& & + g(X,A\xi)g(JY,A\xi) - g(Y,A\xi)g(JX,A\xi) \notag \\
& & - 2g(\xi,AN)g(X,A\xi)\eta(Y) + 2g(X,AN)g(\xi,A\xi)\eta(Y)  \notag \\
\label{eq:sect4:eq1}
& &  + 2g(\xi,AN)g(Y,A\xi)\eta(X) - 2g(Y,AN)g(\xi,A\xi)\eta(X).
\end{eqnarray}
At each point $z \in M$ we again choose $A \in {\mathfrak A}_z$, orthogonal vectors $Z_1,Z_2 \in V(A)$ and $0 \leq t \leq \frac{\pi}{4}$ so that 
\[ N = \cos(t)Z_1 + \sin(t)JZ_2 \]
holds (see \cite[Proposition 3]{R}). Note that the quantities $t$, $A$ and $Z_k$ depend on $z\in M$. 
Because of $\xi = -JN$ we have
\begin{eqnarray*}
N & = & \cos (t) Z_1 + \sin (t)JZ_2, \\
AN & = & \cos (t)Z_1 - \sin (t)JZ_2, \\
\xi & = & \sin (t)Z_2 - \cos (t)JZ_1, \\
A\xi & = & \sin (t)Z_2 + \cos (t)JZ_1.
\end{eqnarray*}
This implies $g(\xi,AN) = 0$ and hence we get from Equation~\eqref{eq:sect4:eq1}
\begin{eqnarray*}
0 & = & 2g(S \phi SX,Y) - \alpha g((\phi S + S\phi)X,Y) + 2 g(\phi X,Y) \\
& & - g(X,AN)g(Y,A\xi) + g(Y,AN)g(X,A\xi)\\
& & + g(X,A\xi)g(JY,A\xi) - g(Y,A\xi)g(JX,A\xi)\\
& &  + 2g(X,AN)g(\xi,A\xi)\eta(Y) - 2g(Y,AN)g(\xi,A\xi)\eta(X).
\end{eqnarray*}
We have $JA\xi = -AJ\xi = - AN$, and inserting this into the previous equation implies the statement of the lemma.
\end{proof}
\par
\vskip 6pt
The preceding formula will be applied both here and in Section~\ref{section 5} 
to get more information on Hopf hypersurfaces with constant mean curvature for which the normal vector field is ${\mathfrak A}$-principal everywhere.
As in Section~\ref{section 2}, we denote by $\bar{\nabla}$ the canonical covariant derivative of ${Q^m}^*$, and by $\bar{\nabla}^{\End}$ the induced covariant
derivative on the endomorphism bundle $\End(T{Q^m}^*)$.
\par
\vskip 6pt
\begin{lm}\label{lemma 4.5}
Let $M$ be a Hopf hypersurface in the complex hyperbolic quadric $\HQ$, $m \geq 3$, such that the
normal vector field $N$ is ${\mathfrak A}$-principal everywhere. Let $A$ be the section of the $S^1$-bundle $\mathfrak{A}$ so that $AN=N$ holds.
Then we have the following:
\begin{itemize}
\item[(i)] ${\bar{\nabla}}^{\End}_XA=0$ for any $X{\in}{\mathcal C}$.
\item[(ii)] $ASX=SX$ for any $X{\in}{\mathcal C}$.
\end{itemize}
\end{lm}
\begin{proof}
Let $q:T{Q^m}^* \to \R$ be the real-valued 1-form on ${Q^m}^*$ so that
\begin{equation}
\label{eq:lemma45:q}
\bar{\nabla}^{\End}_X A = q(X)\cdot JA \quad \text{holds for every $X \in T{Q^m}^*$,}
\end{equation}
see the end of Section~\ref{section 2}. Then let us differentiate
the equation $g(AN,JN)=0$ along any $X{\in}T_xM$, $x{\in}M$. Thereby we obtain
\begin{equation*}
\begin{split}
0=&g(({\bar{\nabla}}^\End_XA)N+A{\bar{\nabla}}_XN, JN)+g(AN,({\bar{\nabla}}^\End_XJ)N+J{\bar{\nabla}}_XN)\\
=&q(X)-g(ASX,JN)-g({\xi},SX) \; ;
\end{split}
\end{equation*}
for the second equals sign it was used that $\bar{\nabla}^\End J = 0$ holds because ${Q^m}^*$ is K\"ahlerian. This gives us for the $1$-form $q$ 
\begin{equation}\label{e41}
q(X)=-g(ASX,{\xi})+g({\xi},SX)=g(S{\xi},X)+g({\xi},SX)=2{\alpha}{\eta}(X),
\end{equation}
where we have used that because of $N \in V(A)$, we have $A\xi = -AJN = JAN = JN = -\xi$.
It follows from Equation~\eqref{e41} that $q(X)=0$ holds for any $X{\in}{\mathcal C}$, whence (i) follows.
\vskip 6pt
\par
Second, we differentiate the formula $AJN=-JAN=-JN$ along the distribution $\mathcal C$. By applying Equation~\eqref{eq:lemma45:q} and again $\bar{\nabla}^{\End}J=0$, we obtain
for \,$X\in \mathcal{C}$\,
$$q(X)JAJN-AJSX=JSX.$$
Because of (i), we have $q(X)=0$, and therefore $-AJSX =JSX$, which implies $ASX=SX$, completing the proof of (ii).
\end{proof}

\medskip

\section{Proof of Main Theorem}\label{section 5}
\par
\vskip 6pt
Let $M$ be an oriented real hypersurface of the complex hyperbolic quadric ${Q^m}^*$ with $m\geq 3$, and let $N$ be a unit normal field of $M$ in ${Q^m}^*$.
In the sequel, we will use the notations introduced in Section~\ref{section 4},
in particular the distributions $\mathcal{C}$ and $\mathcal{Q}$ defined there. As was explained in the Introduction, $M$ is an almost contact manifold, and
$(\phi,\xi,\eta,g)$ is an almost contact metric structure for $M$, where $\phi$ is the \emph{structure tensor field} defined by letting $\phi X$ be the orthogonal projection of $JX$ to $\mathcal C$
for any vector field $X$ on $M$, $\xi := -JN$ is the \emph{Reeb vector field} on $M$, and
$\eta(X)=g(X,\xi)$. We denote by $\nabla$ the Levi-Civita derivative on $M$ induced by the Riemannian metric $g$. 

If $M$ is a contact hypersurface, then
\begin{equation}
\label{eq:s5:phiS}
\phi S + S \phi = k \cdot \phi
\end{equation}
holds, see \cite[Equation~(2.13)]{O2}. Here $k$ is a non-zero constant because of $m\geq 3$, see \cite[Proposition~2.5]{BerndtSuh}. 
In this situation $M$ is automatically a \emph{Hopf hypersurface}, meaning that $\xi$ is a principal curvature direction. Indeed, Equation~\eqref{eq:s5:phiS} shows that $S$ leaves
$\R \xi = \ker \phi$ invariant. The principal curvature function corresponding to $\xi$ is denoted by $\alpha := g(S\xi,\xi)$.

If $X\in \mathcal{C}$ is another principal curvature direction, say for the principal curvature $\sigma$, then Equation~\eqref{eq:s5:phiS} shows that $\phi X$ is a principal curvature
direction for the principal curvature $k-\sigma$. Therefore the mean curvature of $M$ is given by
\begin{equation}
\label{eq:s5:H}
H = \tr(S) = \alpha + (m-1)\cdot k \; .
\end{equation}
It follows that if the contact hypersurface $M$ has constant mean curvature, then the function $\alpha$ is constant.

From here on, we let $M$ be a contact hypersurface, and use the notations introduced above for this situation, throughout the section.
The objective of this section is to prove the Main Theorem. In the first part of the section we will show that the unit normal field $N$ to $M$ is $\mathfrak{A}$-principal
if $M$ has constant mean curvature (Theorem~\ref{Theorem 5.5});
for this purpose we need Lemmas~\ref{lemma 5.2}--\ref{lemma 5.4}; one consequence is that the results of Lemma~\ref{lemma 4.5} are applicable. We can then
complete the classification. 

\begin{lm}\label{lemma 5.2}
Let $M$ be a contact hypersurface in the complex hyperbolic quadric ${Q^m}^*, m\geq 3$.
Then we have for any vector field $X$ on $M$ and any section $A$ of $\mathfrak{A}$
\begin{equation*}
\begin{split}
2S^2 X &=\Big[\eta(X)\big\{2\alpha^2 - \alpha k - 2\big\} - 2 \big\{g(\phi X, AN)\eta(A \xi) -\eta(AN)g(\phi X, A\xi)\big\}\Big]\xi \\
&\quad+2k SX -(\alpha k -2)X + g(\phi X, AN)(A\xi)^T - g(\phi X, A\xi)(AN)^T \\
&\quad+g(\phi X, A\xi)(JA\xi)^T - \{g(X, A\xi) -\eta(X) \eta(A\xi)\}(A\xi)^T \; , 
\end{split}
\end{equation*}
where $(A\xi)^T$, $(AN)^T$, and $(JA\xi)^T$ denote the tangential component of the vector fields $A\xi$, $AN$ and $JA\xi$, respectively.
\end{lm}

\begin{proof}\label{Proof}
It follows from the formula in Lemma~\ref{lemma 4.2} that we have
\begin{equation}\label{5.2}
\begin{split}
0&=2\{-g(S^2 \phi X, Y) + k g(S \phi X, Y)\} - \alpha k g(\phi X, Y) + 2 g(\phi X, Y) \\
&\quad -g(X, AN) g(Y, A\xi) + g(Y, AN)g(X, A\xi) + g(X, A\xi)g(JY, A\xi)\\
&\quad -g(Y, A\xi)g(JX, A\xi)-2g(\xi, AN)g(X, A\xi)\eta(Y) +2g(X, AN)g(\xi, A\xi) \eta(Y)\\
&\quad +2g(\xi, AN)g(Y, A\xi)\eta(X) -2g(Y, AN)g(\xi, A\xi)\eta(X)\;.
\end{split}
\end{equation}
Therefore we have
\begin{equation}\label{5.3}
\begin{split}
&2S^2 \phi X - 2k S\phi X + (\alpha k - 2) \phi X - 2\{g(X, AN) g(\xi, A\xi) - g(\xi, AN) g(X, A\xi)\}\xi \\
&= -g(X, AN)(A\xi)^T + g(X, A\xi)(AN)^T - g(X, A\xi)(JA\xi)^T - g(JX, A\xi)(A\xi)^T \\
&\quad +2\{g(\xi, AN)\eta(X)(A\xi)^T - g(\xi, A\xi)\eta(X)(AN)^T\} \; . 
\end{split}
\end{equation}
By replacing $X$ by $\phi X$ in this equation, we obtain
\begin{equation*}
\begin{split}
&2 S^{2}\phi^2 X - 2k S \phi^2 X + (\alpha k -2)\phi^2 X - 2\{g(\phi X, AN)g(\xi, A\xi) -g(\xi, AN) g(\phi X, A\xi)\}\\
&=-g(\phi X, AN)(A\xi)^T + g(\phi X, A\xi)(AN)^T - g(\phi X, A\xi)(JA\xi)^T - g(J\phi X, A\xi)(A\xi)^T \; . 
\end{split}
\end{equation*}
This equation can be rearranged as follows:
\begin{equation}\label{5.4}
\begin{split}
2S^2 X &= 2\eta(X)S^2\xi + 2k SX - 2k \eta(X) S\xi - (\alpha k -2)X + \eta(X)(\alpha k -2)\xi \\
        &\quad -2\{g(\phi X, AN) g(\xi, A\xi) - g(\xi, AN)g(\phi X, A\xi)\}\xi \\
        &\quad +g(\phi X, AN)(A\xi)^T - g(\phi X, A\xi)(AN)^T + g(\phi X, A\xi)(JA\xi)^T \\
        &\quad +g(J\phi X, A\xi)(A\xi)^T\\
&=\eta(X)\{2\alpha^2 - 2\alpha k + (\alpha k -2)\}\xi + 2k SX - (\alpha k -2) X\\
        &\quad -2\{g(\phi X, AN)g(\xi, A\xi) - g(\xi, AN)g(\phi X, A\xi)\} \xi \\
        &\quad +g(\phi X, AN)(A\xi)^T - g(\phi X, A\xi)(AN)^T + g(\phi X, A\xi)(JA\xi)^T \\
        &\quad +g(J\phi X, A\xi)(A\xi)^T \; ,
\end{split}
\end{equation}
completing the proof of the lemma.
\end{proof}

\begin{lm}\label{lemma 5.3}
For a contact hypersurface $M$ in the complex hyperbolic quadric ${Q^m}^*$, $m\geq 3$, we have for any vector field $X$ on $M$ and any section $A$ of $\mathfrak{A}$
\begin{equation*}
\begin{split}
\sum_{i=1}^{2m-1}  g( (\nabla_{E_i} S)\phi X, E_i) &= \sum_{i=1}^{2m-1} g((\nabla_{\phi X} S)E_i, E_i)\\
&\quad -g(A\phi X, (AN)^T) + g(\phi X, AN)({\rm Tr} A-g(AN, N))\\
&\quad -g(JA\phi X, (A\xi)^T) - g(\phi X, A\xi)g(A\xi, N),
\end{split}
\end{equation*}
Here $(E_i)_{i=1,\dotsc,2m-1}$ is an orthonormal frame field of $M$, we put $E_{2m}:= N$ and ${\rm Tr} A=\sum_{i=1}^{2m} g(AE_i, E_i)$.
\end{lm}

\begin{proof}
We first note that we have the following formulas for the differentiation of $\phi$ (which follow from $\bar{\nabla}^{\End}J=0$), where $X$ and $Y$ are arbitrary vector fields on $M$:
\begin{equation}
\label{eq:s5:phi-diff}
({\nabla}_Y{\phi})X={\eta}(X)SY-g(SY,X){\xi}
\quad\text{and}\quad
({\nabla}_Y{\phi})SX={\eta}(SX)SY-g(SY,SX){\xi} \; .
\end{equation}
We now differentiate Equation~\eqref{eq:s5:phiS} in the direction of a vector field $X$ on $M$, giving
\begin{equation*}
(\nabla_Y \phi)SX + \phi (\nabla_Y S)X + (\nabla_Y S)\phi X + S(\nabla_Y \phi)X = (Yk)\phi X + k(\nabla_Y \phi)X.
\end{equation*}
Because of Equations~\eqref{eq:s5:phi-diff}, we obtain
\begin{equation}\label{5.5}
\begin{split}
&\eta(X)\{S^2 Y + \alpha SY - kSY\} - g(S^2 X + \alpha SX - kSX, Y)\xi \\
&+\phi(\nabla_Y S)X + (\nabla_Y S)\phi X = (Yk)\phi X .
\end{split}
\end{equation}
By substituting Equation \eqref{5.3} from the proof of Lemma~\ref{lemma 5.2} into Equation~\eqref{5.5}, and again using Equations~\eqref{eq:s5:phi-diff}, we obtain
\begin{equation}\label{5.6}
\begin{split}
&\eta(X)\Big[\alpha SY - \frac{\alpha k+2}{2}Y + \Big\{\eta(Y)(\alpha^2 - \frac{\alpha k}{2} +1)\\
&+\big[g(\phi Y, AN)\eta(A\xi)-\eta(AN)g(\phi Y, A\xi)\big]\Big\}\xi \\
&-\frac{1}{2}g(\phi Y, AN)(A\xi)^T + \frac{1}{2} g(\phi Y, A\xi)(AN)^T\\
&- \frac{1}{2}g(\phi Y, A\xi)(JA\xi)^T+\frac{1}{2}\big\{g(Y, A\xi) -\eta(Y)\eta(A\xi)\big\}(A\xi)^T\Big] \\
&-g(\alpha SX - \frac{\alpha k+2}{2}X, Y)\xi \\
&-g(\Big\{\eta(X)(\alpha^2 - \frac{\alpha k}{2} +1) + g(\phi X, AN)\eta(A\xi)\\
&- \eta(AN)g(\phi X, A\xi)\Big\}\xi, Y)\xi+\frac{1}{2}g(g(\phi X, AN)(A\xi)^T \\
&- g(\phi X, A\xi)(AN)^T + g(\phi X, A\xi)(JA\xi)^T \\
&- \big\{g(X, A\xi)-\eta(X)\eta(A\xi)\big\}(A\xi)^T, Y)\xi \\
&+\phi(\nabla_Y S)X + (\nabla_Y S)\phi X = (Yk)\phi X.
\end{split}
\end{equation}
Now we use the orthonormal frame field $(E_i)_{i=1,\dotsc,2m-1}$ on $TM$ to contract Equation (\ref{5.6}), giving
\begin{equation}\label{eq:s5:lem53-a}
\begin{split}
\sum_{i=1}^{2m-1}(E_{i}k){\phi}E_i &= \alpha S\xi - \frac{\alpha k+2}{2}\xi + ({\alpha}^2-\frac{{\alpha}k}{2}+1)\xi  \\
&\quad-\alpha \sum_{i=1}^{2m-1}g(SE_i, E_i)\xi + \frac{\alpha k+2}{2}\sum_{i=1}^{2m-1}g(E_i, E_i)\xi - (\alpha^2 - \frac{\alpha k}{2} +1)\xi  \\
&\quad+\frac{1}{2} \Big\{g(\phi (A\xi)^T, AN) -g(\phi(AN)^T, A\xi)\Big\}\xi  \\
&\quad+\frac{1}{2} g(\phi(JA\xi)^T, A\xi)\xi  \\
&\quad-\frac{1}{2} \Big\{g((A\xi)^T, A\xi) -\eta((A\xi)^T)\eta(A\xi)\Big\}\xi  \\
&\quad+\phi(\nabla_{E_i} S)E_i + (\nabla_{E_i}S)\phi E_i.
\end{split}
\end{equation}
On the other hand, we have
\begin{equation*}
(JA\xi)^T=\phi(A\xi)^T-g(A\xi, N)\xi
\end{equation*}
because of
\,$JA\xi = J((A\xi)^T + (A\xi)^N) = \phi (A\xi)^T + \eta((A\xi)^T)N - g(A\xi, N)\xi$\,. 
This gives the following formula
\begin{equation*}
\phi(JA\xi)^T = \phi^2(A\xi)^T=-(A\xi)^T+\eta((A\xi)^T)\xi.
\end{equation*}
Inserting this formula into Equation~\eqref{eq:s5:lem53-a} yields
\begin{equation}\label{5.7}
\begin{split}
\sum_{i=1}^{2m-1}(E_ik)\phi E_i &= \alpha S\xi - \frac{\alpha k-2}{2}\xi - \alpha\sum_{i=1}^{2m-1}g(SE_i, E_i)\xi \\
&\quad+ \frac{\alpha k-2}{2} \sum_{i=1}^{2m-1} g(E_i, E_i)\xi \\
&\quad -\frac{1}{2}\{g(\phi(A\xi)^T, AN) - g(\phi(AN)^T, A\xi)\}\xi \\
&\quad +\big \{g((A\xi)^T,A\xi) -\eta((A\xi)^T)\eta(A\xi) \big \}\xi \\
&\quad+\sum_{i=1}^{2m-1}\phi (\nabla_{E_i}S)E_i + \sum_{i=1}^{2m-1}(\nabla_{E_i}S)\phi E_i.
\end{split}
\end{equation}

As a consequence of the equation of Codazzi, we have 
\begin{equation*}
\sum_{i=1}^{2m-1} g(\phi(\nabla_{E_i}S)E_i, X) = -\sum_{i=1}^{2m-1} g((\nabla_{E_i} S)E_i, \phi X) = -\sum_{i=1}^{2m-1} g(E_i, (\nabla_{E_i}S)\phi X).
\end{equation*}

Moreover, by putting $X=E_i$, $Y=\phi X, Z=E_i$ into the equation of Codazzi, we obtain
\begin{equation*}
\begin{split}
&\quad \sum_{i=1}^{2m-1} g((\nabla_{E_i}S)\phi X - (\nabla_{\phi X}S)E_i, E_i)\\
&= -\sum_{i=1}^{2m-1} \{g(E_i, AN)g(A\phi X, E_i) - g(\phi X, AN) g(AE_i, E_i)\}\\
&\quad - \sum_{i=1}^{2m-1}\{g(E_i, A\xi)g(JA\phi X, E_i) - g(\phi X, A\xi)g(JAE_i, E_i)\}\\
&=-g(A\phi X, (AN)^T) + g(\phi X, AN)\{{\rm Tr} A- g(AN,N)\}\\
&\quad -g(JA\phi X, (A\xi)^T) + g(\phi X, A\xi)\sum_{i=1}^{2m-1} g(JAE_i, E_i),
\end{split}
\end{equation*}
where we have used the following formulas
\begin{equation*}
\begin{split}
& \sum_{i=1}^{2m-1} g(AE_i, E_i) = {\rm Tr} A - g(AN, N),\\
& \sum_{i=1}^{2m-1} g(AE_i, E_i) = -\sum_{i=1}^{2m-1} g(AE_i, JE_i) = -\sum_{i=1}^{2m-1} g(AE_i, \phi E_i + \eta(E_i)N)\\
& \hspace*{2.8cm} =-\sum_{i=1}^{2m-1} g(AE_i, \eta(E_i)N) = -g(A\xi, N),
\end{split}
\end{equation*}
and
\begin{equation*}
{\rm Tr} A=\sum_{i=1}^{2m-1} g(AE_i, E_i) + g(AN, N).
\end{equation*}
By plugging the preceding formulas into Equation~\eqref{5.7}, the proof of the lemma is completed. 
\end{proof}

\begin{lm}\label{lemma 5.4}
Let $M$ be a contact hypersurface in complex hyperbolic quadric $\HQ$, $m\geq 3$. 
Then we have for any vector field $X$ on $M$ and any section $A$ of $\mathfrak{A}$
\begin{equation*}
\begin{split}
\sum_{i=1}^{2m-1} g((\nabla_{E_i} S)\phi E_i, X) & = \sum_{i=1}^{2m-1} g((\nabla_X S)E_i, \phi E_i) + 2(m-1)\eta(X)\\
&\quad -g(AX, A\xi)-g(AX, \xi)g(N,AN)+g(\xi, AN)g(AX,N)\\
&\quad -g(AX, A\xi)+g(AX, N)g(N,A\xi)+\eta(A\xi)g(AX,\xi)\\
&\quad +g(X, A\xi)\big\{{\rm Tr} A-g(AN,N) - \eta(A\xi)\big\} \;,
\end{split}
\end{equation*}
where again $(E_i)_{i=1,\dotsc,2m-1}$ is an orthonormal frame field of $M$, we put $E_{2m}:= N$ and ${\rm Tr} A=\sum_{i=1}^{2m} g(AE_i, E_i)$.
\end{lm}

\begin{proof}
We apply the Codazzi equation once more, with $X=E_i, Y=X$ and $Z=\phi E_i$, obtaining
\begin{equation}\label{eq:s5:l54-proof-eq}
\begin{split}
&\quad \sum_{i=1}^{2m-1} \{g((\nabla_{E_i} S)\phi E_i, X) -g((\nabla_X S) E_i, \phi E_i)\} \\
&= \sum_{i=1}^{2m-1} g((\nabla_{E_i}S)X-(\nabla_X S)E_i, \phi E_i)\\
&= \sum_{i=1}^{2m-1} \{\eta(E_i)g(\phi X, \phi E_i) -\eta(X)g(\phi E_i, \phi E_i) -2\eta(\phi E_i)g(\phi E_i, X)\}\\
&\quad+\sum_{i=1}^{2m-1}g(E_i, AN)g(AX, \phi E_i) - \sum_{i=1}^{2m-1}g(X, AN)g(AE_i, \phi E_i)\\
&\quad+\sum_{i=1}^{2m-1}g(E_i, A\xi)g(JAX, \phi E_i) - \sum_{i=1}^{2m-1}g(X, A\xi)g(JAE_i, \phi E_i).
\end{split}
\end{equation}
Here, the terms on the right-hand side can be evaluated as follows:
\begin{equation*}
\begin{split}
&\quad -\sum_{i=1}^{2m-1} g(E_i, AN)g(AX, \phi E_i) = -\sum_{i=1}^{2m-1} g(E_i, AN) g(AX, JE_i - \eta(E_i)N)\\
&= -\sum_{i=1}^{2m-1} g(E_i, AN) g(AX,JE_i) + g(\xi, AN)g(AX, N)\\
&= -g(AX,JAN)+g(AX, JN)g(N,AN)+g(\xi, AN)g(AX, N)\\
&= -g(AX, A\xi)-g(AX, \xi)g(N,AN)+g(\xi, AN)g(AX, N),
\end{split}
\end{equation*}

\begin{equation*}
\begin{split}
&\quad \  -\sum_{i=1}^{2m-1} g(X, AN) g(AE_i, \phi E_i) = -\sum_{i=1}^{2m-1} g(X, AN) g(AE_i, JE_i - \eta(E_i)N)\\
&=-\sum_{i=1}^{2m} g(X, AN) g(AE_i, JE_i) + g(X, AN) g(AN, JN) \\
&\quad \ +\sum_{i=1}^{2m-1} g(X, AN) \eta(E_i) g(AE_{i}, N)\\
&=g(X, AN) g(AN, JN) + g(X, AN) g(A\xi, N)=0,
\end{split}
\end{equation*}

\begin{equation*}
\begin{split}
&\quad \sum_{i=1}^{2m-1} g(E_i, A\xi) g(JAX, \phi E_i) = \sum_{i=1}^{2m-1} g(E_i, A\xi) g(JAX, JE_i - \eta((E_i)N)\\
&=\sum_{i=1}^{2m-1} g(E_i, A\xi) g(JAX, JE_i) -\eta(A\xi)g(JAX, N)\\
&=g(JAX, JA\xi) - g(JAX, JN) g(N, A\xi) -\eta(A\xi)g(JAX, N)\\
&=g(AX, A\xi)-g(AX, N)g(N,A\xi)-\eta(A\xi)g(AX, \xi),\\
\end{split}
\end{equation*}

and
\begin{equation*}
\begin{split}
&\quad  \sum_{i=1}^{2m-1} g(JAE_i, \phi E_i) = -\sum_{i=1}^{2m-1} g(AE_i, J\phi E_i)\\
&=\sum_{i=1}^{2m-1} g(AE_i, \phi^2E_i + \eta(\phi E_i)N)\\
&=-\sum_{i=1}^{2m-1} g(AE_i, -E_i +\eta(E_i)\xi) = \sum_{i=1}^{2m-1} g(AE_i, E_i) -\sum_{i=1}^{2m-1}\eta(E_i)g(AE_i,\xi)\\
&= {\rm Tr}A - g(AN,N) -\eta(A\xi).
\end{split}
\end{equation*}

By applying these evaluations to Equation~\eqref{eq:s5:l54-proof-eq} we complete the proof of the Lemma.
\end{proof}

\begin{thm}\label{Theorem 5.5}
Let $M$ be a contact hypersurface in the complex hyperbolic quadric $\HQ$, $m \geq 3$. If the mean curvature of $M$ is constant, then the unit normal vector field $N$ of $M$ is  $\mathfrak{A}$-principal, that is, $AN=N$ for some section $A$ of the $S^1$-bundle $\mathfrak{A}$.
\end{thm}

\begin{proof}
We let the section $A$ of  $\mathfrak{A}$ be at first arbitrary. We let $(E_i)_{i=1,\dotsc,2m-1}$ be an orthonormal frame of $M$ as before, but now we suppose $E_{2m-1}=\xi$. We also put $E_{2m}=N$ again. 
By taking the inner product of Equation~\eqref{5.7} with a given tangent vector field $X$ on $M$ and using Equation~\eqref{eq:s5:H}, and Lemmas~\ref{lemma 5.3} and \ref{lemma 5.4}, we get
\begin{equation}\label{5.8}
\begin{split}
-(\phi X)k &=\sum_{i=1}^{2m-1}(E_i k)g(\phi E_i, X)\\
&=\alpha^2 \eta(X) - \frac{\alpha k -2}{2}\eta(X) - \alpha\big\{{\alpha}+(m-1)k\big\}\eta(X)\\
&\quad +\frac{{\alpha}k-2}{2}(2m-1)\eta(X)\\
&\quad -\frac{1}{2} \big \{g(\phi(A\xi)^T, AN) - g(\phi (AN)^T, A\xi) \big \} \eta(X)\\
&\quad +\big\{g((A\xi)^T, A\xi)-\eta((A\xi)^T)\eta(A\xi)\big\}\eta(X)\\
&\quad-\sum_{i=1}^{2m-1}g((\nabla_{\phi X}S)E_i, E_i) + g(A\phi X, (AN)^T)\\
&\quad -g(\phi X, AN)\big\{{\rm Tr} A-g(AN, N)\big\}\\
&\quad +g(JA\phi X, (A\xi)^T )+ g(\phi X, A\xi)g(A\xi, N)\\
&\quad+\sum_{i=1}^{2m-1} g((\nabla_X S)E_i, \phi E_i) + 2(m-1)\eta(X)\\
&\quad -g(AX, A\xi) - g(AX, \xi)g(AN, N) + g(\xi, AN)g(AX, N)\\
&\quad -g(AX, A\xi) + g(AX, N)g(N, A\xi) + \eta(A\xi)g(AX, \xi)\\
&\quad +g(X, A\xi)\big\{{\rm Tr} A - g(AN, N) - \eta(A\xi)\big\},
\end{split}
\end{equation}
where ${\rm Tr} A=\sum_{i=1}^{2m}g(AE_i, E_i)$  denotes the trace of the complex conjugation $A$ in $T_z {\HQ}$, $z \in {\HQ}$.

Because $M$ has constant mean curvature, we have
$$ \sum_{i=1}^{2m-1}g(({\nabla}_{\phi X}S)E_i,E_i)=0 \quad\text{and}\quad {\rm Tr}({\nabla}_XS){\phi}=\sum_{i=1}^{2m-1}g(({\nabla}_XS)E_i,{\phi}E_i)=0 \; , $$
and moreover the following formulas hold:
$$ \sum_{i=1}^{2m-2}g(Ae_i, e_i)=0 \quad\text{and}\quad g(A{\xi},N)=0 \; . $$
By replacing $X$ with $\phi X$ in Equation~\eqref{5.8} and applying the preceding formulas, we obtain
\begin{equation}\label{5.9}
\begin{split}
0&=g(A\phi^2 X, (AN)^T) - g(\phi^2 X, AN)g(A\xi,\xi) + g(JA\phi^2 X, (A\xi)^T)\\
&\quad + g(\phi^2 X, A\xi) g(A\xi, N)-g(A\phi X, \xi)g(AN, N) + \eta(A\xi)g(A\phi X, \xi).
\end{split}
\end{equation}
The terms on the right-hand side of (\ref{5.9}) can be calculated as follows:
\begin{equation*}
\begin{split}
&g(A\phi^2 X, (AN)^T) = g(A(-X+\eta(X)\xi), (AN)^T)\\
                      &\hspace*{3cm}= -g(AX,(AN)^T) + \eta(X)g(A\xi, (AN)^T)\\
                      &\hspace*{3cm}= -g(AX, AN-g(AN,N)N) + \eta(X) g(A\xi, AN-g(AN,N)N)\\
                      &\hspace*{3cm}=g(AN,N)g(AX,N),\\
&g(\phi^2 X, AN)  = g(-X+\eta(X)\xi, AN) = -g(X, AN),\\
&g(JA{\phi}^2X, (A\xi)^T) = g(JA\phi^2 X, A\xi - g(A\xi, N)N)\\
                      &\hspace*{2.9cm}=g(JA(-X+\eta(X)\xi), A\xi)\\
                      &\hspace*{2.9cm}=-g(JAX, A\xi) +\eta(X)g(JA\xi, A\xi)\\
                      &\hspace*{2.9cm}=-g(AX, AJ\xi)=-g(AX, AN)=0,\\
&g(A\phi X, \xi) =-g(A\phi X, JN) = -g(AJ \phi X, N) \\
               &\hspace*{1.8cm} =-g((A(\phi^{2}X + \eta(\phi X)N), N)\\
               &\hspace*{1.8cm} = g(AX, N) - \eta(X)g(A\xi, N)=g(AX,N).
\end{split}
\end{equation*}
Substituting these formulas into Equation~\eqref{5.9} gives
\begin{equation}\label{eq:s5:thm55-eq}
\begin{split}
0 & = g(AN,N)g(AX,N) + g(X,AN)g(A\xi,\xi)\\
  & \quad \ \  - g(AX,N)g(AN,N) + \eta(A\xi)g(AX,N)\\
  & = 2g(A\xi,\xi)g(AN,X)
\end{split}
\end{equation}
for any vector field $X$ on $M$.

We now fix $z\in M$ and separate three possible cases:
\begin{itemize}
\item[(1)] $g(A_z\xi_z,\xi_z)\neq 0$ for some $A_z \in \mathfrak{A}_z$.
\item[(2)] $g(A_z\xi_z,\xi_z)= 0$ for all $A_z \in \mathfrak{A}_z$, but there exists a sequence $(z_n)_{n\in \mathbb{N}}$ of points $z_n \in M$ converging to $z$, so that for every $n\in \mathbb{N}$
  there exists $A_{z_n} \in \mathfrak{A}_{z_n}$ with $g(A_{z_n}\xi_{z_n},\xi_{z_n})\neq 0$.
\item[(3)] $g(A_{z'}\xi_{z'},\xi_{z'})=0$ for all $z' \in M$ in some neighborhood of $z$, and all $A_{z'} \in \mathfrak{A}_{z'}$.
\end{itemize}

In case (1), Equation~\eqref{eq:s5:thm55-eq} gives $g(A_zN_z,X_z)=0$ for every $X_z \in T_zM$, hence $A_zN_z$ is another unit normal vector to $M$ at $z$. 
Because $M$ is a hypersurface in $\HQ$, it follows that either $A_zN_z=N_z$ and then $N_z \in V(A_z)$, or else $A_zN_z=-N_z$ and then $N_z \in V(-A_z)$ holds. In either case, $N_z$ is $\mathfrak{A}$-principal.

In case (2), $N_{z_n}$ is $\mathfrak{A}$-principal for every $n\in \mathbb{N}$ by case (1), and it follows that $N_z$ is also $\mathfrak{A}$-principal by continuity reasons.

We will now complete the proof of the theorem by showing that case (3) leads to a contradiction, and therefore cannot occur.
We fix any local section $A$ of $\mathfrak{A}$ near $z$, so that by the case hypothesis we have
\begin{equation}
\label{eq:s5:thm55-Axi}
g(A\xi,\xi)=g(JA\xi,\xi) = 0 \;;
\end{equation}
because of $N = J\xi$, these equations also imply
\begin{equation}
\label{eq:s5:thm55-AN}
g(AN,N) = g(JAN,N) = 0 \; .
\end{equation}
We also note $g(A\xi,N)=g(JAN,N)=0$, which shows that $A\xi$ is tangential to $M$. Moreover, $A\xi$ is orthogonal to $\xi$ by Equation~\eqref{eq:s5:thm55-Axi}, and therefore
$A\xi$ is in fact a section of $\mathcal{C}$. 

By differentiation of the equation $g(A\xi,\xi)=0$ in the direction of any (local) vector field $X$ of $M$, we obtain
\begin{align}
0 & = g((\bar{\nabla}^\End_X A)\xi,\xi) + g(A\bar{\nabla}_X \xi,\xi) + g(A\xi,\bar{\nabla}_X \xi) \notag \\
\label{eq:s5:thm55-iso-eq1}
& = g((\bar{\nabla}^\End_X A)\xi,\xi) + 2\,g(A\bar{\nabla}_X \xi,\xi) \; .
\end{align}
We again consider the 1-form $q(X)$ defined by Equation~\eqref{eq:lemma45:q}; by that equation and Equation~\eqref{eq:s5:thm55-Axi} we then have
\begin{equation}
\label{eq:s5:thm55-iso-eq1a}
g((\bar{\nabla}^\End_X A)\xi,\xi) = q(X) \cdot g(JA\xi,\xi) = 0 \; .
\end{equation}
Moreover, by the Gauss equation of first order we have (where $h$ denotes the second fundamental form of the immersion $M \hookrightarrow \HQ$)
\begin{align*}
\bar{\nabla}_X \xi & = \nabla_X \xi + h(X,\xi) = -\nabla_X (JN) + g(SX,\xi)\,N \\
& = -\phi\nabla_X N + g(\alpha\,X,\xi)\,N = -\phi SX + \alpha\,\eta(X)\,N \; .
\end{align*}
Therefore we have
\begin{align}
g(A\bar{\nabla}_X \xi,\xi) & = -g(A\phi SX,\xi) + \alpha \,\eta(X)\,g(AN,\xi) = -g(A\phi SX,\xi) = -g(\phi SX,A\xi) \notag \\
\label{eq:s5:thm55-iso-eq1b}  
& = g(S\phi X,A\xi) - k\cdot g(\phi X,A\xi) = g(\phi X, SA\xi - k\cdot A\xi)\; . 
\end{align}
where $g(AN,\xi)=g(N,A\xi)=0$ follows from the fact that $A\xi$ is tangential, and the second line follows from Equation~\eqref{eq:s5:phiS}. By plugging Equations~\eqref{eq:s5:thm55-iso-eq1a}
and \eqref{eq:s5:thm55-iso-eq1b} into Equation~\eqref{eq:s5:thm55-iso-eq1}, we obtain
$g(\phi X, SA\xi - k\cdot A\xi) = 0$,
which shows that the orthogonal projection of $SA\xi-k\cdot A\xi$ onto $\mathcal{C}$ vanishes. In fact, $A\xi$ is a section of $\mathcal{C}$, and because $S$ leaves $(\R \xi)^\perp = \mathcal{C}$
invariant, $SA\xi$ also is a section of $\mathcal{C}$. Hence we have
\begin{equation}
\label{eq:s5:thm55-iso-eq2}
SA\xi = k\cdot A\xi \; .
\end{equation}

We now differentiate the equation $g(AN,JN)=0$, which follows from Equation~\eqref{eq:s5:thm55-AN}, in the direction of $X$, and obtain by Equations~\eqref{eq:lemma45:q}, \eqref{eq:s5:thm55-AN}
and \eqref{eq:s5:thm55-iso-eq2}
\begin{align*}
0 & = g((\bar{\nabla}^\End_X A)N,JN) + g(A\bar{\nabla}_X N,JN) + g(AN,J\bar{\nabla}_XN) \\
& = q(X)\cdot g(JAN,JN) - g(ASX,JN) -g(AN,JSX) \\
& = 2 g(ASX,\xi) = 2 g(X,SA\xi) = 2k\cdot g(X,A\xi) \; .
\end{align*}
By choosing the tangent vector field $X=A\xi$ in the last equation, we obtain $k=0$, which is a contradiction.
\end{proof}
\par
\vskip 6pt

After these preparations we are ready to calculate the principal curvatures of a given contact hypersurface $M$ with constant mean curvature in the complex hyperbolic quadric ${Q^m}^*$.
This result will finally lead to the proof of the Main Theorem stated in the Introduction.

\par
\vskip 6pt
\begin{lm}\label{lemma 5.6}
Let $M$ be a contact hypersurface in the complex hyperbolic quadric $\HQ$, $m{\ge}3$ with constant mean curvature.
Then $M$ has three constant principal curvatures $\alpha$, $\lambda=0$ and $\mu$ with the following eigenspaces and multiplicities:
\begin{center}
\begin{tabular}{|l|l|l|}
\hline
\mbox{principal curvature} & \mbox{eigenspace}  & \mbox{multiplicity}\\
\hline
${\alpha}$ & ${\mathbb R}\xi$ & $1$ \\
${\lambda}=0$ & $J(V(A) \ominus \mathbb{R}N)$ & $m-1$ \\
${\mu}$ & $V(A) \ominus {\mathbb R}N$ & $m-1$\\
\hline
\end{tabular}
\end{center}
Here $A$ is the section of $\mathfrak{A}$ so that $AN=N$ holds. 
If we let $k$ be the constant from Equation~\eqref{eq:s5:phiS} and choose the sign of the unit normal field $N$ so that $k>0$, 
there exists $r> 0$ so that the values of the principal curvatures $\alpha$ and $\mu$ are as follows:
\begin{itemize}
\item[(i)] If $k<\sqrt{2}$, then $\alpha=\sqrt{2}\,\coth{\sqrt 2}r$, $\mu=\sqrt{2}\,\tanh{\sqrt 2}r$.
\item[(ii)] If $k=\sqrt{2}$, then $\alpha=\mu=\sqrt{2}$.
\item[(iii)] If $k>\sqrt{2}$, then $\alpha=\sqrt{2}\,\tanh{\sqrt 2}r$, $\mu=\sqrt{2}\,\coth{\sqrt 2}r$.
\end{itemize}
\end{lm}
\par
\vskip 6pt

\begin{proof}
From Theorem~\ref{Theorem 5.5} we know that the unit normal $N$ is $\mathfrak{A}$-principal, that is, $AN=N$ for some section $A$ of the $S^1$-bundle $\mathfrak{A}$.
Moreover, because $M$ has constant mean curvature, it follows from Equation~\eqref{eq:s5:H} that   
the function $\alpha=g(S\xi,\xi)$ is constant. Note that $\alpha$ is a principal curvature of $M$, and that $\xi$ is the corresponding
principal curvature vector field.

Now let $X$ be a principal curvature vector field of $M$ that is orthogonal to $\xi$, i.e.~$X$ is a section of the vector subbundle $\mathcal{C}$ of $TM$, and let
$\sigma$ be the corresponding principal curvature function on $M$. 
Lemma~\ref{lemma 4.2} then gives
$${\alpha}({\phi}S+S{\phi})X=2S{\phi}SX + 2{\phi}X \; . $$
Because of $SX=\sigma X$ and Equation~\eqref{eq:s5:phiS}, this yields
\begin{equation}
\label{eq:s5:l56-eq1}
{\alpha}k{\phi}X=2{\sigma}S{\phi}X + 2{\phi}X \; .
\end{equation}
On the other hand, Lemma~\ref{lemma 4.2} also implies the formula
\begin{equation}
\label{eq:s5:l56-eq2}
(2{\sigma}-{\alpha})S{\phi}X=({\alpha}{\sigma}-2){\phi}X \; .
\end{equation}
Let $z\in M$ be given. If $2\sigma(z)\neq \alpha$, then Equation~\eqref{eq:s5:l56-eq2} implies $S\phi X_z = \tfrac{\alpha\,\sigma(z)-2}{2\sigma(z)-\alpha}\phi X_z$,
and therefore Equation~\eqref{eq:s5:l56-eq1} gives
$$ {\alpha}k{\phi}X_z=2\sigma(z)\, \frac{\alpha\,\sigma(z)-2}{2\sigma(z)-\alpha}\phi X_z + 2{\phi}X_z \; , $$
and hence because of $\phi X_z \neq 0$
$$ {\alpha}k =2\sigma(z)\, \frac{\alpha\,\sigma(z)-2}{2\sigma(z)-\alpha} + 2 \; , $$
which shows that $\sigma(z)$ is a solution of the quadratic equation
\begin{equation}
\label{eq:s5:l56-quadreq}
2\sigma^2 - 2k\sigma + \alpha k - 2 = 0 \; . 
\end{equation}
On the other hand, if $2\sigma(z)=\alpha$, then we have $(\alpha\,\sigma(z)-2)\phi X_z=0$ by Equation~\eqref{eq:s5:l56-eq2}, whence it follows that $\alpha\,\sigma(z)=2$ because
of $X_z\neq 0$. Therefore we have $\sigma(z)=\pm 1$ and $\alpha = 2\sigma(z)$, which implies that $\sigma(z)$ is a solution of the quadratic equation \eqref{eq:s5:l56-quadreq}
also in this case.

Thus we have shown that the principal curvature function $\sigma$ is a solution of the quadratic equation \eqref{eq:s5:l56-quadreq} with constant coefficients. Therefore
$\sigma$ is constant. It also follows that there exist at most two different such principal curvatures corresponding to principal curvature vectors in $\mathcal{C}$. 
If we denote these two principal curvatures by $\lambda$ and $\mu$, then $\mu = k-\lambda$ holds.

We will now show that one of the principal curvatures $\lambda$ and $\mu$ is zero.
In fact, we have $ASX = SX$ for any section $X$ of $\mathcal{C}$ by Lemma~\ref{lemma 4.5}(ii).
This shows that for any non-zero principal curvature of $M$, the corresponding principal curvature vector subbundle of \,$\mathcal{C}$ is contained in
$V(A)$. If both $\lambda$ and $\mu$ were non-zero, then in fact all of the complex vector bundle $\mathcal{C}$ would be contained in $V(A)$, which is a contradiction to $JA=-AJ$. Therefore
one of the principal curvatures must be zero, say $\lambda=0$. It then follows from Equation~\eqref{eq:s5:l56-quadreq} that $\alpha k -2 = 0$, hence $\alpha=\tfrac{2}{k}$,
and $\mu=k$ holds. It also follows that the principal curvature subbundle for the principal curvature $\mu\neq 0$ is $V(A) \cap \mathcal{C} = V(A) \ominus \R N$,
and that the principal curvature subbundle for the principal curvature $\lambda =0$ is the ortho-complement of the former subbundle in $\mathcal{C}$, and hence equal to
$J(V(A) \ominus \R N)$. 

Now suppose that the sign of the normal field $N$ is chosen so that $k>0$. If then
$k<\sqrt{2}$ holds, there exists $r>0$ with $k={\sqrt 2}\tanh{\sqrt 2}r=\mu$ and therefore
$\alpha=\tfrac{2}{k}={\sqrt 2}\coth{\sqrt 2}r$. If $k=\sqrt{2}$, we have
$\mu=k=\sqrt{2}$ and $\alpha=\tfrac{2}{k}=\sqrt{2}$.  If $k>\sqrt{2}$ holds,
there exists $r> 0$ with $k={\sqrt 2}\coth{\sqrt 2}r=\mu$
and therefore $\alpha=\tfrac{2}{k}={\sqrt 2}\tanh{\sqrt 2}r$.
\end{proof}
\medskip

We are now ready for the proof of the Main Theorem stated in the Introduction.

\medskip

\begin{proof}[Proof of the Main Theorem.]
Let $M$ be a connected orientable contact real hypersurface with constant mean curvature in the complex hyperbolic quadric ${Q^{m}}^*$ with the constant $k$ defined as in Equation~\eqref{eq:s5:phiS}.
The function $\alpha=g(S\xi,\xi)$ is constant by Equation~\eqref{eq:s5:H} and our assumption of constant mean curvature. Moreover, the unit normal $N$ is $\mathfrak{A}$-principal by Theorem~\ref{Theorem 5.5}.
We let $A$ be the section of $\mathfrak{A}$ so that $AN=N$ holds.   
We also choose the sign of the unit normal $N$ such that $k>0$. From Lemma~\ref{lemma 5.6} we then know that $M$ has exactly three principal curvatures $\alpha$, $\lambda$ and $\mu$,
they are constant and have the values given in Lemma~\ref{lemma 5.6} (depending on the value of $k$).
If we put $T_{\rho}=\{X\in{\mathcal C}{\vert}SX={\rho}X\}$ for $\rho \in \{\alpha,\lambda,\mu\}$, we have
$$ T_\alpha = \R \xi\;,\quad T_\lambda = J(V(A) \ominus \R N) \quad\text{and}\quad T_\mu = V(A) \ominus \R N \;, $$
in particular ${\mathcal C}=T_{\lambda}{\oplus}T_{\mu}$.
\par
\vskip 6pt
For given $p\in M$ and $r>0$, let us denote by ${\gamma}_p$ the geodesic in ${Q^{m}}^*$ with ${\gamma}_p(0)=p$ and ${\dot{\gamma}}_p(0)=N_p$, and by $F$ the smooth map
$$F: M \longrightarrow {{Q}^{m}}^*, \quad p \longmapsto {\gamma}_p(r).$$
\par
Geometrically, $F$ is the displacement of $M$ at distance $r$ in the direction of the normal vector field $N$. For each point $p \in M$ the differential $d_pF$ of $F$ at $p$ can be computed by using Jacobi vector fields. More specifically, we have
$$d_pF(X)=Z_X(r),$$
where $Z_X$ is the Jacobi vector field along ${\gamma}_p$ with initial values $Z_X(0)=X$ and $Z_X'(0)=SX$. Then it follows from the explicit description of the curvature tensor of ${Q^m}^*$ given in Section~\ref{section 2} that the normal Jacobi operator ${\bar R}_N$ is given by
\begin{equation*}
\begin{split}
{\bar R}_NZ=&{\bar R}(Z,N)N\\
 =&Z+AZ-2g(Z,N)N+2g(Z,JN)JN
\end{split}
\end{equation*}
for any vector field $Z$ on $M$. It follows that the normal Jacobi operator ${\bar R}_N$ has two eigenvalues $0$ and $-2$ with corresponding eigenspaces
$J(V(A){\ominus}{\mathbb{R}}N)=T_\lambda$ and $(V(A){\ominus}{\mathbb{R}}N){\oplus}{\mathbb{R}}JN=T_\mu\oplus T_\alpha$ respectively.
\par
\vskip 6pt
Therefore the Jacobi vector field $Z_X$ along $\gamma_p$ is given explicitly as follows:
\begin{equation*}
Z_X(r)=\begin{cases}
(\cosh({\sqrt 2}r) - \frac{\alpha}{\sqrt 2}\sinh({\sqrt 2}r))E_X(r) & \text{if $X \in T_{\alpha}=\mathbb{R}JN$,}\\
(\cosh({\sqrt 2}r) - \frac{\mu}{\sqrt 2}\sinh({\sqrt 2}r))E_X(r) & \text{if $X \in T_{\mu}$,}\\
E_X(r) & \text{if $X \in T_{\lambda}$,}
\end{cases}
\end{equation*}
where $E_X$ denotes the parallel vector field along ${\gamma}_p$ with $E_X(0)=X$. It follows from this explicit description that the dimension of the kernel of the linear map $T_pM \to T_{F(p)}{Q^m}^*,\,X \mapsto d_pF(X)=Z_X(r)$
does not depend on $p$. Hence $F$ has constant rank, and is therefore a submersion onto a submanifold $P$ of ${Q^m}^*$. The vector $\eta_p := \dot{\gamma}_p(r)$\, is a unit normal vector to $P$ at
the point $F(p)$, every unit normal vector of $P$ is obtained in this way, and the shape operator $S_{\eta_p}$ of $P$ with respect to the normal vector $\eta_p$ is given by
$$ S_{\eta_p} Z_X(r) = Z_X'(r) \; . $$
We also note that because the complex structure $J$ and the $S^1$-bundle of real structures $\mathfrak{A}$ of ${Q^m}^*$ are parallel, if $(\ker d_pF)^\perp$ has one of the
properties complex, $\mathfrak{A}$-invariant, totally real, $\mathfrak{A}$-principal, then $P$ will have the same property.

To complete the proof of the Main Theorem, we now distinguish three cases.

\emph{Case~1: $k < \sqrt{2}$.} By Lemma~\ref{lemma 5.6} we then have $\alpha=\sqrt{2}\,\coth{\sqrt 2}r$ and $\mu=\sqrt{2}\,\tanh{\sqrt 2}r$ for some $r>0$. This gives $Z_{JN}(r)=0$ and
$Z_{X}(r)\neq 0$ for all $X\in (T_\lambda\oplus T_\mu)\setminus \{0\}$. It follows that
$\text{ker}\ d_pF=\mathbb{R}JN$, hence $(\ker d_pF)^\perp=\mathcal{C}_p=\mathcal{Q}_p$ (see Lemma~\ref{lemma 4.1}(i)). Therefore the focal manifold $P$ is an \,$\mathfrak{A}$-invariant,
$(m-1)$-dimensional complex submanifold of ${Q^m}^*$. We now show that the shape operator $S_{\eta_p}$ vanishes. In fact we have for \,$X\in T_\lambda$\,
$$ S_{\eta_p}Z_X(r) = Z_X'(r) = 0 $$
and for $X \in T_\mu$
\begin{align*}
S_{{\eta}_p}Z_X(r)=Z_X'(r)
&=({\sqrt 2}\sinh{\sqrt 2}r - {\mu}\cosh{\sqrt 2}r)E_X(r)\\
&=({\sqrt 2}\sinh{\sqrt 2}r - {\sqrt 2}\tanh{\sqrt 2}r\cosh{\sqrt 2}r)E_X(r)
=0.
\end{align*}
Therefore $P$ is a totally geodesic submanifold of ${Q^m}^*$. By applying the classification theorem for totally geodesic submanifolds in the complex quadric $Q^m$ \cite{K}, \cite{K2} to its non-compact dual
${Q^m}^*$, we know that $P$ is locally congruent to the complex $(m-1)$-dimensional complex hyperbolic quadric described in Section~\ref{section 3}, and $M$ is a tube of radius $r$ around $P$.

\emph{Case~2: $k = \sqrt{2}$.} By Lemma~\ref{lemma 5.6} we then have $\alpha=\mu=\sqrt{2}$, and therefore for any $r>0$
$$ Z_X(r) = \begin{cases}
  e^{-\sqrt{2}\,r}\,E_X(r) & \text{if $X \in T_\alpha \oplus T_\mu$} \\
  E_X(r) & \text{if $X\in T_\lambda$}
\end{cases} \; . $$
Therefore $M$ does not have a focal set in this case. On the other hand, the Jacobi vector fields $Z_X(r)$ remain bounded for $r \to +\infty$. This means that all the normal geodesics
emanating from $M$ are asymptotic to each other, i.e.~they meet at a single point of the ``ideal boundary'' (set of points at infinity) of the non-compact manifold ${Q^m}^*$; this point
is given by an equivalence class of asymptotic geodesics whose tangent vectors are all $\mathfrak{A}$-principal.
Hence $M$ is a horosphere with this point as center at infinity.

\emph{Case~3: $k > \sqrt{2}$.} We then have $\alpha=\sqrt{2}\,\tanh{\sqrt 2}r$ and $\mu=\sqrt{2}\,\coth{\sqrt 2}r$ for some $r>0$ by Lemma~\ref{lemma 5.6}. This gives us
$Z_X(r)=0$ for all $X \in T_\mu$, and $Z_X(r)\neq 0$ for all $X \in (\mathbb{R}JN \oplus T_\lambda)\setminus \{0\}$. It follows that
$\ker d_pF = T_\mu$, hence $(\ker d_pF)^\perp = JV(A)$ is an $m$-dimensional, $\mathfrak{A}$-principal, totally real subspace of $T_p{Q^m}^*$. Therefore the focal manifold $P$
is an $m$-dimensional, $\mathfrak{A}$-principal, totally real submanifold of ${Q^m}^*$. An analogous calculation as in Case~1 shows that $P$ is totally geodesic.
By again applying the classification of totally geodesic submanifolds in $Q^m$ \cite{K}, \cite{K2} to the non-compact dual ${Q^m}^*$, we see that $P$ is locally congruent
to the totally real $m$-dimensional real hyperbolic space described in Section~\ref{section 3}, and $M$ is a tube of radius $r$ around $P$.
\end{proof}

\begin{acknowledgements}
{\rm This work was done while the first author was visiting professor at the Research Institute of Real and Complex Submanifolds in Kyungpook National University during October, 2017.}
\end{acknowledgements}

\end{document}